\documentclass[12pt,a4paper]{article}
\usepackage[english,french]{babel}
\usepackage[T1]{fontenc}
\usepackage{latexsym,amsfonts,amssymb,amsmath,amsthm,mathrsfs, mathtools}
\usepackage{tikz}
\usepackage[all]{xy}
\usepackage{ulem}
\usepackage{tabularx}
\usepackage{times}
\usepackage{pdflscape}
\usepackage{graphicx}
\usepackage{float}
\usepackage[hidelinks, backref=page]{hyperref}
\usepackage{geometry}
\renewcommand*{\backref}[1]{}
\renewcommand*{\backrefalt}[4]{%
        \ifcase #1 (Not cited.)%
        \or        (Cited on page~#2.)%
        \else      (Cited on pages~#2.)%
        \fi}

\newtheorem{Theorem}{Theorem}
\newtheorem{Lemma}{Lemma}
\newtheorem{Proposition}{Proposition}

\newtheorem{Corollary}{Corollary}
\newtheorem{Remark}{Remark}

\def\N{\mathbb{N}}
\def\Z{\mathbb{Z}}

\def\A{\mathcal{A}}

\def\H{\mathrm{H}}

\def\E{\mathrm{E}}
\def\F{\mathbb{F}_{2}}
\def\FF{\mathbb{F}}

\def\S{\mathbb{S}}
\def\ovk{\overline{\kappa}}
\def\Ext{\mathrm{Ext}}
\def\EG24{E_C^{hG_{24}}}
\def\hocolim{\mathrm{hocolim}}
\def\Gal {\mathrm{Gal}}
\begin{document}
\selectlanguage{english}
\title{On the $tmf$-Hurewicz image of $A_1$ }
\author{Viet-Cuong Pham
}
\date{\today}
\maketitle

\abstract{Let $A_1$ be any spectrum in the class of finite spectra whose mod-2 cohomology is isomorphic to $\A(1)$ as a module over the subalgebra $\A(1)$ of the Steenrod algebra; let $tmf$ be the connective spectrum of topological modular forms. In this paper, we prove that the $tmf$-Hurewicz image of $A_1$ is surjective.}
\let\thefootnote\relax\footnote{This work was partially supported by the ANR Project ChroK, ANR-16-CE40-0003 }
\tableofcontents

\section*{Introduction}
In \cite{DM81}, Davis and Mahowald constructed a class of connective finite spectra whose mod $2$ cohomology is isomorphic to $\A(1)$ as a module over the subalgebra $\A(1)$ generated by $Sq^1, Sq^2$ of the Steenrod algebra. This class of spectra has four different homotopy types, denoted by $A_1[ij]$ for $i,j\in \{0,1\}$, see the introduction of \cite{BEM17} for an explanation of this use of notation. We write $A_1$ to refer to any of the latter and call each of them a version of $A_1$. The spectrum $A_1$ is constructed via three cofiber sequences, starting from the sphere spectrum $S^0$, as follows. Let $V(0)$ be the mod $2$ Moore spectrum, i.e., the cofiber of multiplication by $2$ on $S^0$. Next, let $Y$ be the cofiber of multiplication by $\eta$, the first Hopf element, on $V(0)$. Then, Davis and Mahowald show that $Y$ admits a $v_1$-self maps, $v_1: \Sigma^2 Y\rightarrow Y$. Then $A_1$ is the cofiber of any of these $v_1$-self maps of $Y$. There are eight homotopy classes of $v_1: \Sigma^2 Y\rightarrow Y$, giving rise to four different homotopy types of $A_1$. We note further that the spectra $A_1[00]$ and $A_1[11]$ are Spanier-Whitehead self-dual and $A_1[10]$ and $A_1[01]$ are Spanier-Whitehead dual to each other.
\\\\
Let $tmf$ be the ring spectrum of connective topological modular forms. This spectrum is constructed using a certain sheave in ring spectra on the \'etale site of the moduli stack of elliptic curves, hence the name, see \cite{DFHH14} for the construction. The spectrum $tmf$ plays an important role in investigating chromatic level $2$ in chromatic homotopy theory. Let $H$ denote the $tmf$-Hurewicz map of $A_1$, i.e.,
$H: A_1\rightarrow tmf\wedge A_1$, given by smashing $A_1$ with the unit of $tmf$. The goal of this paper is to study the induced map in homotopy of $H$:
\begin{equation}\label{hurew-map}H_* : \pi_*(A_1)\rightarrow \pi_*(tmf\wedge A_1)
\end{equation}
 Closely related to the homomorphism (\ref{hurew-map}) is the edge homomorphism of the topological duality spectral sequence aiming at analysing the $K(2)$-localisation of $A_1$, where $K(2)$ is the second Morava $K$-theory at the prime $2$. This is, in fact, our initial motivation for studying (\ref{hurew-map}), see Section \ref{TDSS} for a further discussion. On another perspective, the map (\ref{hurew-map}) is the edge homomorphism of the $tmf$-based Adams spectral sequence for $A_1$. This is an upper half plane spectral sequence converging to $\pi_*(A_1)$ starting with the $\E_1$-term:
$$\E_1^{n, t} = \pi_{t}(tmf^{\wedge (n+1)}\wedge A_1)\Longrightarrow \pi_{t-n}(A_1).$$
 The main theorem of this paper is the following.
\begin{Theorem}\label{tmf-Hurc} The $tmf$-Hurewicz homomorphism $\pi_*(A_1)\rightarrow \pi_*(tmf\wedge A_1)$ is surjective for all versions of $A_1$.
\end{Theorem}
\noindent
As a consequence, we have the following theorem (see Theorem \ref{edge-proof}).
\begin{Theorem} \label{edge-homo}The edge homomorphism of the topological duality resolution $$\pi_*(E_C^{h\S^1_2}\wedge A_1)\rightarrow \pi_*(E_C^{hG_{24}}\wedge A_1)$$ is surjective.
\end{Theorem}
\noindent
Let us discuss another consequence of Theorem \ref{tmf-Hurc}. The spectrum $A_1$ is a finite spectrum of type $2$ (i.e. $K(1)_*(A_1) = 0$, where $K(1)$ is the first Morava $K$-theory and $K(2)_*(A_1)\ne 0$), hence admits a $v_2$ self-map. In fact, the authors of \cite{BEM17} show that $A_1$ admits a $v_2$ self-map of periodicity $32$, i.e., a map $v_2^{32}: \Sigma^{192} A_1\rightarrow A_1$, that induces an isomorphism in $K(2)$-homology. Let $[(v_2^{32})^{-1}]A_1$ denote the associated telescope, i.e.,
$$[(v_2^{32})^{-1}]A_1 = \hocolim (A_1 \rightarrow \Sigma^{-192} A_1\rightarrow ...\rightarrow \Sigma^{-192k}A_1\rightarrow ...).$$
 Suppose that $x$ is a $v_2$-periodic element of $\pi_t(A_1)$, i.e., $x$ is a nontrivial element of $ \pi_t([(v_2^{32})^{-1}]A_1)$. This means that there is a map $x: S^{t}\rightarrow A_1$ such that the composite 
$$S^{t+192k}\xrightarrow{\Sigma^{192k}x} \Sigma^{192k}A_1\xrightarrow{v_2^{32k}}A_1$$ 
is essential for $k\in \N$. Since the top cell of $A_1$ is in dimension $6$, the composite $v_2^{32k}\circ \Sigma^{192k}x$ gives rise to a nontrivial element of $\pi_*S^0$ in the stem $192k +t - i_{k}$ for some $0\leq i_{k}\leq 6$. As $k$ varies in $\N$, the latter is called a $v_2$-periodic family of $S^0$.
The $K(2)$-localisation of $A_1$ is useful to detect $v_2$-periodic elements of $\pi_*(A_1)$ because the $K(2)$-localisation map $A_1\rightarrow L_{K(2)}A_1$ factors as $A_1\rightarrow [(v_2^{32})^{-1}]A_1\rightarrow L_{K(2)}A_1.$ Moreover, we prove in Theorem 5.3.22 of \cite{Pha18}
\begin{Theorem}[Pha18] The $K(2)$-localisation $tmf\wedge A_1\rightarrow L_{K(2)}(tmf\wedge A_1)$ induces an isomorphism on homotopy groups of non-negative stems. 
\end{Theorem}
\noindent 
Therefore, because of the commutative diagram
$$\xymatrix{ A_1\ar[d] \ar[rr] && tmf\wedge A_1\ar[d]\\
			L_{K(2)}A_1\ar[rr]&& L_{K(2)}(tmf\wedge A_1),
 }$$
 Theorem \ref{tmf-Hurc} asserts that all elements of $\pi_*(tmf\wedge A_1)$ lift to $v_2$-periodic elements of $\pi_*(A_1)$. It is then interesting to be able to locate the corresponding $v_2$-periodic families in $\pi_*(S^0)$.
 %We note also that Ravenel's Telescope Conjecture predicts that the map $[(v_2^{32})^{-1}]A_1\rightarrow L_{K(2)}A_1$ is a homotopy equivalence.\\\\
\\\\
 \textbf{Convention.}
Unless otherwise stated, all spectra are localised at the prime $2$. $\mathrm{H}^{*}(X)$ and $\mathrm{H}_{*}(X)$ denote the mod-2 (co)homology of the spectrum $X$. Given a Hopf algebra $A$ over a field $k$ and $M$ a $A$-comodule, we abbreviate $\Ext_{A}^{*}(k,M)$ by $\Ext_{A}^{*}(M)$.\\\\
\noindent
\textbf{Acknowlegements.} The author would like to thank Hans-Werner Henn for many helpful discussions and for carefully reading an earlier draft of this paper and for suggesting improvements. 
\section{Background and tools}
\subsection{The dual of the Steenrod algebra}
Recall that the Steenrod algebra $\A$ is a co-commutative graded Hopf algebra over $\F$, generated by the Steenrod squares $Sq^{i}$ in degree $i$ for $i\geq 0$, which subject to the Adem relations 
$$Sq^{a}Sq^{b} = \sum\limits_{i=0}^{\lfloor \frac{a}{2} \rfloor} {b-i-1 \choose a-2i}  Sq^{a+b-i}Sq^{i}$$ 
for all $a,b >0$ and $a<2b$. In \cite{Mil58}, Milnor determines the dual  $\A_*$ of the Steenrod algebra, which is a commutative Hopf algebra over $\F$. As an algebra, $\A_*$ is a polynomial algebra generated by $\xi_i$ for $i\geq 0$ in degree $2^{i}-1$ with $\xi_0=1$, i.e., $$\A_*=\F[\xi_1, \xi_2, ...].$$ The coproduct or the diagonal is given by $$\Delta (\xi_k) = \sum\limits_{i=0}^{k}\xi_i^{2^{k-i}}\otimes \xi_{k-i}.$$  Denote by $\zeta_i$ the conjugate of $\xi_i$. Then we have that 
\begin{equation}\label{diagonal A} \Delta(\zeta_k) = \sum\limits_{i=0}^{k}\zeta_{i}\otimes \zeta_{k-i}^{2^{i}}.
\end{equation}
\textbf{Subalgebras of $\A$.} For $n\geq 1$, denote by $\A(n)$ the subalgebra of $\A$ generated by $Sq^{i}$ for $0\leq i\leq n$. The dual $\A(n)_*$ of $\A(n)$ is the quotient of $\A_*$ by the ideal $(\zeta_1^{2^{n+1}}, \zeta_2^{2^{n}}, ...,\zeta_{n}^{4}, \zeta_{n+1}^{2}, \zeta_{n+2}, ...)$, i.e., $$\A(n)_* = \F[\zeta_1, \zeta_2, ..., \zeta_{n+1}]/(\zeta_1^{2^{n+1}}, \zeta_2^{2^{n}}, ...,\zeta_{n}^{4}, \zeta_{n+1}^{2}).$$
\textbf{Notation.} If $A\rightarrow B$ is a map if Hopf algebras, then $A\square_{B}k$ denotes the group of primitives of $A$ viewed as a $B$-comodule via the given map.
\\\\
The main tool used in this paper to prove Theorem \ref{tmf-Hurc} is the Adams spectral sequence, which will be abbreviated by ASS in the sequel. More precisely, one has the following theorem, known to Hopkins and Mahowald and whose proof can be found in \cite{Mat16}.
\begin{Theorem}
There is an isomorphism of $\A_*$-comodule algebras:
$$\H_*tmf \cong \A_*\square_{\A(2)_*}\F.$$
\end{Theorem}
\noindent
As a consequence, if $X$ is a connective spectrum, then the Adams spectral sequence for $tmf\wedge X$ reads as 
$$\E_2^{s,t} \cong \Ext_{\A_*}^{s,t}(\F, (\A_*\square_{\A(2)_*}\F) \otimes \H_*X)\Longrightarrow \pi_{t-s} (tmf\wedge X).$$
By the change-of-rings isomorphism, the $\E_2$-term of this spectral sequence is isomorphic to $\Ext_{\A(2)_*}^{s,t}(\H_*X)$. Furthermore, the $tmf$-Hurewicz map of $A_1$ induces a map of Adams spectral sequences
$$\xymatrix{ \Ext_{\A_*}^{s,t}(\F,\H_*(A_1))\ar[rr]^{H_*}\ar@{=>}[d]&& \Ext_{\A(2)_*}^{s,t}(\F,\H_*(A_1))\ar@{=>}[d]\\
		\pi_{t-s}(A_1)\ar[rr]^{H_*}&& \pi_{t-s}(tmf\wedge A_1)
,}$$
where the map in the $\E_2$-terms is induced the the natural projection of Hopf algebras $\A_*\rightarrow \A(2)_*$.
To analyse $\Ext_{\A_*}^{s,t}(\F,\H_*(A_1))$ as well as $\Ext_{\A(2)_*}^{s,t}(\F,\H_*(A_1))$, we use the Davis-Mahowald spectral sequence, which is reviewed in the next section.

\subsection{The Davis-Mahowald spectral sequence}
Initially, the Davis-Mahowald spectral sequence was used by Davis and Mahowald in \cite{DM82} to calculate $\Ext$-groups over the subalgebra $\A(2)$. In \cite{Pha18}, we established a slight generalisation of the latter. Let us recall this construction. 
\\\\
Let $A$ be a commutative Hopf algebra over a field $k$ of characteristic $2$. Let $E$ and $P$ be the graded exterior algebra and the polynomial algebra on a $k$-vector space $V$ such that $V$ lives in degree $1$, respectively. Let $E_i$ and $P_i$ denote the subspace of elements of homogeneous degree $i$, respectively. Suppose that $E$ has a structure of a $A$-comodule algebra such that $k\oplus V$ is a sub-$A$-comodule of $E$. Since $P_1$ sits in a short exact sequence 
\begin{equation}\label{ExSq}0\rightarrow k\rightarrow k\oplus E_1\xrightarrow{p} P_1,
\end{equation}
$P_1$ admits a unique structure of $A$-comodule making $p$ a map of $A$-comodules. Then $P$ admits a structure of $A$-comodule algebra by the following lemma.
\begin{Lemma} If $P_1^{\otimes n}$ is equipped with the usual structure of $A$-comodule of a tensor product, then $P_n$ admits a unique structure of $A$-comodule making the multiplication $P_1^{\otimes n}\rightarrow P_n$ a map of $A$-comodules.
\end{Lemma}
\begin{proof} The canonical map $P_1^{\otimes n}\rightarrow P_n$ is surjective and its kernel is spanned by element of the form $y_1\otimes y_2\otimes...\otimes y_n - y_{\sigma_1}\otimes y_{\sigma(2)}\otimes...\otimes y_{\sigma(n)}$ where $\sigma$ is a permutation of the set $\{1, 2, ..., n\}$. Then, since $A$ is commutative, we see that the kernel is stable under the coaction of $A$. The lemma follows.
\end{proof}
We define the following cochain complex $(E\otimes P, d)$ with
\begin{itemize}
  \item[i)] $(E\otimes P)_{-1} = k$
  
  \item[ii)] $(E\otimes P)_{m}= E\otimes P_{m}$ for $m\geq 0$ 
  
  \item[iii)] $d : k=(E\otimes P)_{-1}\rightarrow E = (E\otimes P)_{0}$ being the unit of $E$
  
  \item[iv)] $d(\prod\limits_{j=1}^{n} x_{i_{j}}\otimes z) = \sum\limits_{t=1}^{n} \prod\limits_{j\ne t} x_{i_{j}}\otimes p(x_{i_{t}})z$ where $x_{i_{j}}\in E_{1}$, $z\in P_{m}$ and $p$ is the projection of (\ref{ExSq})
\end{itemize} 
Proposition 2.1.4 of \cite{Pha18} shows that $(E\otimes P, d)$ is a differential graded algebra which is an exact sequence of $A$-comodules. As a consequence, there is a spectral sequence of algebras 
\begin{equation} \label{SSA}
\Ext_A^{s}(k, E\otimes P_n)\Longrightarrow \Ext_{A}^{s+n}(k,k).
\end{equation}
Furthermore, if $M$ is an  $A$-comodule, there is a spectral sequence of modules over (\ref{SSA})
\begin{equation} \label{SSM}
\Ext_A^{s}(k, E\otimes P_n\otimes M)\Longrightarrow \Ext_{A}^{s+n}(k,M).
\end{equation}
The Davis-Mahowald spectral sequence, or DMSS for short, comes to this paper in the following form. Let $A$ and $B$ be commutative graded Hopf algebras over $\F$ together with a map of Hopf algebras $A\rightarrow B.$ Then $A$ can be considered as a left $B$-comodule algebra. The group of primitives $A\square_{B}k$ inherits a structure of $A$-comodule algebra from $A$, i.e., the inclusion of subgroup $A\square_{B}k\rightarrow A$ is a map of $A$-comodule algebras. If it turns out that $E:= A\square_{B}k$ is an exterior algebra, then we are often in a situation to construct the Davis-Mahowald spectral sequence, precisely when $E$ is generated by the $k$-vector space $V$, as an exterior algebra and $k\oplus V$ is a sub $A$-comodule of $e$. In this situation, by the change-of-rings theorem, the $\E_1$-term of the latter is isomorphic to 
$$\E_1^{s,t,n}= \Ext_A^{s,t}(k, E\otimes P_n) \cong \Ext^{s,t}_{B}(k, P_n).$$ This allows us to reduce the problem of computing $\Ext$-groups over $A$ to $\Ext$-groups over $B$, which is often simpler.
\\\\
\textbf{Example.} Consider the natural projection of commutative Hopf algebras $$\A(n)_*\rightarrow \A(n-1)_*.$$ Then we have that $$\A(n)_*\square_{\A(n-1)_*}\F\cong E(\zeta_1^{2^{n}}, \zeta_{2}^{2^{n-1}}, ..., \zeta_{n+1} ),$$ the exterior algebra on $\F\{\zeta_1^{2^{n}}, \zeta_{2}^{2^{n-1}}, ..., \zeta_{n+1}\}$. Let us denote $\zeta_k^{2^{n+1-k}}$ by $x_k$. Using Equation (\ref{diagonal A}), we see that the $\A(n)_*$-comodule structure of $\A(n)_*\square_{\A(n-1)_*}\F$ is given by $$\Delta(x_k) = \sum\limits_{i=0}^{i=k}\zeta_i^{2^{n+1-k}}\otimes x_{k-i}$$ where $x_0=1$ by convention. Thus, the Davis-Mahowald spectral sequence reads as $$\E_1^{s,t,n} = \Ext_{\A(n-1)_*}^{s,t}(P_n)\Longrightarrow \Ext_{\A(n)_*}^{s+n,t}(\F).$$
\begin{Remark} This example is an important tool to perform the calculation of $\Ext_{\A(2)_*}^{*,*}(\H_*(A_1))$ in \cite{Pha18}. We will apply the DMSS for other examples apprearing in the proof of Theorem \ref{structure E2}.
\end{Remark}
\section{The $tmf$-Hurewicz image of $A_1$}
The Hurewicz map $A_1\rightarrow tmf\wedge A_1$ is a map of modules over the ring spectrum $A_1\wedge DA_1$, hence $H_*: \pi_*(A_1)\rightarrow \pi_*(tmf\wedge A_1)$
is a map of modules over $\pi_*(A_1\wedge DA_1)$. We recollect some elements of the latter, important to this work. The induced map in homotopy of the unit $S^0\rightarrow A_1\wedge DA_1$ sends $\nu\in \pi_{3}S^0$ and $\ovk\in \pi_{20}S^0$ to nontrivial elements, denoted by the same names. This is because $\nu$ and $\ovk$ are sent nontrivially to $\pi_*(tmf\wedge A_1)$ via the composite $$S^0\rightarrow A_1\wedge DA_1\rightarrow A_1\xrightarrow{H}tmf\wedge A_1,$$ where the middle map is induced by the projection $DA_1\rightarrow S^0$ on the top cell of  $DA_1$. Recall that $A_1$ has a $v_2^{32}$-self-map $v_2^{32}: \Sigma^{192}A_1\rightarrow A_1$. Its adjoint $S^{192}\rightarrow A_1\wedge DA_1$ represents an element of $\pi_{192}(A_1\wedge DA_1)$, also denoted by $v_2^{32}$.
\\\\
Let $(\overline{\kappa},\nu)$ be the ideal of $\pi_*(S^0)$ generated by $\overline{\kappa}$ and $\nu$. Consider the following commutative diagram
$$\xymatrix{ \pi_*(A_1)\ar[rr]^{H_*}\ar[d]&& \pi_*(tmf\wedge A_1)\ar[d]\\
			\pi_*(A_1)/(\overline{\kappa},\nu)\ar[rr]&& \pi_*(tmf\wedge A_1)/(\overline{\kappa},\nu).
}$$
We see that the upper horizontal map is surjective if and only if the lower is surjective. In fact, this is an easy consequence of $\pi_*(A_1)$ being bounded below.
To prove that the lower map is surjective we can proceed as follows. For all $\overline{x}\in \pi_*(tmf\wedge A_1)/(\overline{\kappa},\nu)$, first, lift $\overline{x}$ to an element $x\in \pi_*(tmf\wedge A_1)$, then, find a class that detects $x$ in the ASS for $tmf\wedge A_1$ and finally, show that class lifts to a permanent cycle in the ASS for $A_1$. 
\\\\
The first and the second steps are rather straightforward from our analysis of the ASS for $tmf\wedge A_1$ in \cite{Pha18}. We recall here appropriate information. %The group $\Ext_{\A(2)_*}^{s,t}(\H_*(A_1))$ is computed in \cite{Pha18}, Proposition 3.2.5. 
Let $\F[\nu, g, w_2]/(\nu^3, g\nu)$ be the subalgebra of $\Ext_{\A(2)_*}^{*,*}(\F)$ generated by $\nu$, $g$, $w_2$ with $|\nu| = (1,4)$, $|g| = (4,24)$ and $|w_2| = (8,56)$. The classes $\nu$, $g$ lift to classes of the same name in $\Ext_{\A_*}^{4,24}(\F)$ and converge to $\nu$, $\ovk$, respectively, in the ASS for $S^0$.
\begin{Proposition}[\cite{Pha18}, Proposition 3.2.5]\label{AdamsA_1} As a module over $\F[\nu, g, w_2]/(\nu^3, g\nu)$, the group $\Ext_{\A(2)_*}^{s,t}(\H_*(A_1))$ is a direct sum of cyclic modules generated by the following classes with the respective annihilator ideal
$$\begin{array}{llllllll}
e[0,0]& e[1,5] & e[1,6]& e[2,11]&e[3,15] & e[3,17] & e[4,21] & e[4,23]\\
(0) & (\nu^2) & (0) & (\nu^2)& (\nu^2) & (0) & (\nu^2) & (0)\\
\\
e[6,30]& e[6,32]& e[7,36] & e[7,38] & e[8,42] & e[9,47] & e[9,48] & e[10,53]\\
(\nu)& (\nu) & (\nu) & (\nu) & (\nu) & (\nu) & (\nu) & (\nu),
\end{array}$$
where $e[s,t]$ is the unique non-trivial class of $\Ext_{\A(2)_*}^{s,t+s}(\H_*(A_1))$.
\end{Proposition} 
\noindent
According to Theorem 5.3.22 of \cite{Pha18}, $\pi_*(tmf\wedge A_1)$ is $\Delta^8$-periodic. More precisely, multiplication by $\Delta^8\in \pi_{192}(tmf)$ induces an isomorphism $\pi_k(tmf\wedge A_1)\rightarrow \pi_{k+192}(tmf\wedge A_1)$ for $k\geq 0$. The same property holds for multiplication by $v_2^{32}$.
\begin{Proposition}\label{multi v2} Multiplication by $v_2^{32}\in \pi_{192}(A_1\wedge DA_1)$ induces an isomorphism $\pi_k(tmf\wedge A_1)\rightarrow \pi_{k+192}(tmf\wedge A_1)$ for $k\geq 0$.
\end{Proposition} 
\begin{proof} To simplify the notation, we denote by $v_2^{32}$ and $\Delta^8$ the maps $Id_{tmf}\wedge v_2^{32}$ and $\Delta^8\wedge Id_{A_1}$ from $\Sigma^{192}tmf\wedge A_1$ to $tmf\wedge A_1$, respectively. By using the natural equivalence, $$F_{tmf}(tmf\wedge A_1, tmf\wedge A_1)\simeq tmf\wedge A_1\wedge DA_1,$$ where $F_{tmf}(-,-)$ denotes the function spectrum in the category of $tmf$-modules, it suffices to prove that there is a positive integer $n$ such that $(\Delta^8)^n = (v_2^{32})^{n}\in \pi_{192n}(tmf\wedge A_1\wedge DA_1).$
Since $\Delta^8$ induces an isomorphism $\pi_*(tmf\wedge A_1)\rightarrow \pi_{*+192}(tmf\wedge A_1)$ for $*\geq 0$, multiplication by $\Delta^8$ induces an isomorphism
\begin{equation}\label{iso Delta8}
\pi_*(tmf\wedge A_1\wedge DA_1)\rightarrow \pi_{*+192}(tmf\wedge A_1\wedge DA_1) \ \ \mbox{for}\  *\geq 0.
\end{equation}
By the construction in \cite{BEM17}, a $v_2^{32}$-self map of $A_1$ is detected, in the $\E_2$-term of the ASS for $A_1\wedge DA_1$, by a class that is sent to a class detecting $\Delta^8$ in the ASS for $tmf\wedge A_1\wedge DA_1.$ It means that the difference $\Delta^8-v_2^{32}$ is detected in an Adams filtration greater than $32$, the Adams filtration of $\Delta^8$ and of $v_2^{32}$. Together with the isomorphism (\ref{iso Delta8}), the difference $\Delta^8-v_2^{32}$ is equal to $\Delta^8x$ for some element $x\in \pi_0(tmf\wedge A_1\wedge DA_1)$ detected in a positive filtration of the ASS. As a consequence, $x$ is nilpotent, as the ASS for $tmf\wedge A_1\wedge DA_1$ has a vanishing line parallel to that of the ASS for $tmf\wedge A_1$, see Theorem \ref{Vanishing line}. Furthermore, $\Delta^8-v_2^{32}$ has finite order and $\Delta^8$ is in the center of $\pi_*(tmf\wedge A_1\wedge DA_1)$. Therefore, by using the binomial formula, we see that $(\Delta^8+ v_2^{32}-\Delta^8)^{2^k}$ is equal to $(\Delta^8)^{2^k}$ for $k$ large enough.
\end{proof}
\noindent
As an immediate consequence of the above proposition, multiplication by $v_2^{32}$ induces an isomorphism $\pi_{k}(tmf\wedge A_1)/(\nu,\ovk)\rightarrow \pi_{k+192}(tmf\wedge A_1)/(\nu,\ovk)$ for $k\geq 0$.
Moreover, it follows from the proof of Theorem 5.3.22 of \cite{Pha18} that, we can identify a set of generators of $\pi_*(tmf\wedge A_1)/(\ovk,\nu)$ for $0\leq *< 192$ as $\Z_2$-modules. Therefore, the latter generates $\pi_*(tmf\wedge A_1)/(\ovk,\nu)$ as a $\Z_2[v_2^{32}]$-module.
We give in Table (\ref{List M}) and Table (\ref{List N}), a list of generators of the non self-dual versions $A_1[00]$ and $A_1[11]$ and in Table (\ref{List P}) and Table (\ref{List Q}), a list of generators of the self-dual versions $A_1[01]$ and $A_1[10]$, respectively. This distinction is because the proof that they lift to permanent cycles in the ASS for $A_1$ are different, see Proposition \ref{petit stem}, \ref{non-self dual}, \ref{self dual}. We denote by $M$, $N$, $P$, $Q$ the set of generators listed in Table (\ref{List M}), Table (\ref{List N}), Table (\ref{List P}), Table (\ref{List Q}), respectively. In these tables, the pairs of integers indicate the bidegree $(t,s)$ of the corresponding generators living in $\Ext^{s,t+s}$ and we switch to the notation $e_{t}$ instead of $e[s,t]$ to denote a generator in bidegree $(t,s)$.

\begin{table}[H]
$$\begin{array}{llllllll}
 (0,0)&(5,1)&(6,1)&(11,2)&(15,3)&(17,3)&(21,4)&(23,4)\\
 e_0&e_5& e_6&e_{11}&e_{15}&e_{17}&e_{21}&e_{23}\\
 \\
(30,6)&(32,6)&(36,7)&(38,7)&(42,8)&(47,9)&(48,9)&(53,10)\\
e_{30}&e_{32}&e_{36}&e_{38}&e_{42}&e_{47}&e_{48}&e_{53}\\
\\
(48,8)&(53,9)&(54,9)&(59,10)&(63,11)&(65,11)&(69,12)&(71,12)\\
 w_2e_0&w_2e_5& w_2e_6&w_2e_{11}&w_2e_{15}&w_2e_{17}&w_2e_{21}&w_2e_{23}\\
 \\
(78,14)&(80,14)&(84,15)&(86,15)&(90,16)&(95,17)&(96,17)&(101,18)\\
w_2e_{30}&w_2e_{32}&w_2e_{36}&w_2e_{38}&w_2e_{42}&w_2e_{47}&w_2e_{48}&w_2e_{53}\\
\end{array}$$
\caption{List M. Generators of $\pi_*(tmf\wedge A_1)/(\overline{\kappa},\nu)$ as $\F[\Delta^8]$-module for the non-self dual versions $A_1[00]$ and $A_1[11]$. }
\label{List M}
\end{table}
\begin{table}[H]
$$\begin{array}{llllllll}
 (99,17)&(104,18)&(105,18)&(110,19)&(114,20)&(116,20)\\
\nu w_2^2 e_0&\nu w_2^2e_5&\nu w_2^2e_6&\nu w_2^2e_{11}&\nu w_2^2e_{15}&\nu w_2^2e_{17}\\
\\
(120,21)&(122,21)&(147,25)&(152,26)&(153,26)\\
\nu w_2^2e_{21}&\nu w_2^2e_{23}&\nu w_2^3e_0&\nu w_2^3e_5&\nu w_2^3e_6\\
\\
(158,27)&(162,28)&(164,28)&(168,29)&(170,29)\\
\nu w_2^3e_{11}&\nu w_2^3e_{15}&\nu w_2^3e_{17}&\nu w_2^3e_{21}&\nu w_2^3e_{23}
\end{array} $$
\caption{List N. Generators of $\pi_*(tmf\wedge A_1)/(\overline{\kappa},\nu)$ as $\F[\Delta^8]$-module for the non-self dual versions $A_1[00]$ and $A_1[11]$.}
\label{List N}
\end{table}
%%%%%%%%%%%%%%%%%%%%%%%%
\begin{table}[H]
$$\begin{array}{llllllll}
 (0,0)&(5,1)&(6,1)&(11,2)&(15,3)&(17,3)&(21,4)&(23,4)\\
 e_0&e_5& e_6&e_{11}&e_{15}&e_{17}&e_{21}&e_{23}\\
 \\
(30,6)&(32,6)&(36,7)&(38,7)&(42,8)&(47,9)&(48,9)\\
e_{30}&e_{32}&e_{36}&e_{38}&e_{42}&e_{47}&e_{48}\\
\\
(48,8)&(53,10)&(53,9)&(54,9)&(59,10)&(63,11)&(65,11)\\
 w_2e_0&e_{53}&w_2e_5& w_2e_6&w_2e_{11}&w_2e_{15}&w_2e_{17}\\
 \\
(69,12)&(74,13)&(78,14)&(80,14)&(84,15)&(90,16)&(95,17)\\
w_2e_{21}&\nu w_2e_{23}&w_2e_{30}&w_2e_{32}&w_2e_{36}&w_2e_{42}&w_2e_{47}\\
\end{array}$$
\caption{List P. Generators of $\pi_*(tmf\wedge A_1)/(\overline{\kappa},\nu)$ as $\F[\Delta^8]$-module for the self dual versions $A_1[01]$ and $A_1[10]$.}
\label{List P}
\end{table}
\begin{table}[H]
$$\begin{array}{llllllll}
 (96,16)&(101,17)&(105,18)&(110,19)&(111,19)&(116,20)\\
w_2^2 e_0&w_2^2e_5&\nu w_2^2e_6&\nu w_2^2e_{11}&w_2^2e_{15}&\nu w_2^2e_{17}\\
\\
(120,21)&(122,21)&(126,22)&(147,25)&(152,26)&(153,26)\\
\nu w_2^2e_{21}&\nu w_2^2e_{23}&w_2^2e_{30}&\nu w_2^3e_0&\nu w_2^3e_5&\nu w_2^3e_6\\
\\
(158,27)&(162,28)&(164,28)&(168,29)&(170,29)\\
\nu w_2^3e_{11}&\nu w_2^3e_{15}&\nu w_2^3e_{17}&\nu w_2^3e_{21}&\nu w_2^3e_{23}
\end{array} $$
\caption{List Q. Generators of $\pi_*(tmf\wedge A_1)/(\overline{\kappa},\nu)$ as $\F[\Delta^8]$-module for the self dual versions $A_1[01]$ and $A_1[10]$.}
\label{List Q}
\end{table}

 We now proceed to prove that all classes in these tables lift to permanent cycles in the Adams SS for $A_1$. There are two main steps to this end. In the first place, we show that the induced map on the $\E_2$-terms of the Hurewicz map $$H_*: \Ext_{\A_*}^{*,*}(\H_*A_1)\rightarrow \Ext_{\A(2)_*}^{*,*}(\H_*A_1)$$ is surjective. This implies, in particular, that the classes in $M \cup N$ (respectively $P\cup Q$) lift to the $\E_2$-term of the ASS for $A_1$. In the second place, we show that the $\Ext_{\A_*}^{*,*}(\H_*A_1)$ has a certain structure (see Theorem \ref{structure E2 A_1}) that allows us to rule out non-trivial differentials on lifts of the classes in $M\cup N$ (respectively $P\cup Q$) in the ASS for $A_1$. 
\subsection{The algebraic $tmf$-Hurewicz homomorphism}
\begin{Theorem}\phantomsection\label{Vanishing line}(Vanishing line) For $n\geq 0$ or $n=\infty$, $\Ext_{\A(n)_*}^{s,t} (\H_*A_1)$ has vanishing line $t-s < f(s)$, where $f(s) = 5s- 4$ if $s\leq 6$ and $f(s) = 5s$ if $s>6$, i.e.,
$$\Ext_{\A(n)_*}^{s,t} (\H_*A_1) =0\ \mbox{if}\ t-s < f(s).$$
\end{Theorem}
\begin{proof} 
The statement for $n=0,1$ follows from the fact that $\H_*A_1$ is $\A(0)_*$- and $\A(1)_*$-cofree. The statement for $n=2$ follows from the explicit structure of $\Ext_{\A(2)_*}^{s,t}(\H_*A_1)$. Now suppose $n\geq 3$.
Set $\Gamma = \A(n)_*\square_{\A(2)_*}\F$: $\Gamma$ is an $\A(n)_*$-comodule algebra. The unit $\F\rightarrow \A(n)_*\square_{\A(2)_*}\F$ is a map of $\Gamma$-comodules; denote by $\overline{\Gamma}$ the quotient of the latter; so that we have the short exact sequence of $\A(n)_*$-comodules
$$0\rightarrow \overline{\Gamma}^{\otimes r}\rightarrow \Gamma\otimes \overline{\Gamma}^{\otimes {r}}\rightarrow \overline{\Gamma}^{\otimes r+1}\rightarrow 0,$$ for $r\geq 0$. Slicing these together, we get a long exact sequence of $\A(n)_*$-comodules
$$0\rightarrow\F\rightarrow \Gamma\rightarrow \Gamma\otimes\overline{\Gamma}\rightarrow...\rightarrow \Gamma\otimes \overline{\Gamma}^{\otimes r}\rightarrow...$$
which gives rise to a spectral sequence converging to $\Ext_{\A(n)_*}^{*,*}(\H_*A_1)$ with $\mathrm{E}_1$-term isomorphic to $\Ext_{\A(n)_*}^{*,*}( \Gamma\otimes \overline{\Gamma}^{\otimes r}\otimes \H_*A_1)$:
\begin{equation}\phantomsection\label{vanishing SS}\mathrm{E}_1^{s,t,r}=\Ext_{\A(n)_*}^{s,t}( \Gamma\otimes \overline{\Gamma}^{\otimes r}\otimes \H_*A_1)\Longrightarrow \Ext_{\A(n)_*}^{s+r,t}(\F,\H_*A_1).
\end{equation}
By the change-of-rings isomorphism, $$\Ext_{\A(n)_*}^{s,t}( \Gamma\otimes \overline{\Gamma}^{\otimes r}\otimes \H_*A_1)\cong \Ext_{\A(2)_*}^{s,t}( \overline{\Gamma}^{\otimes r}\otimes \H_*A_1).$$ 
We see that $\overline{\Gamma}^{\otimes r}$ is $(8r-1)$-connected because $\overline{\Gamma}$ is $7$-connected. Together with the fact that 
$$\Ext_{\A(2)_*}^{s,t}(\H_*A_1) = 0\ \mbox{if}\ t-s<f(s),$$ we obtain that $$\Ext_{\A(2)_*}^{s,t}(\overline{\Gamma}^{\otimes r}\otimes \H_*A_1) = 0$$
 if $t-s<f(s) + 8r $ or equivalently if $t-(s+r)<f(s+r) + 2r$. We can now conclude by using the spectral sequence (\ref{vanishing SS}).
\end{proof}
\begin{Theorem}\phantomsection\label{Approximation}(Approximation Theorem) Let $m\geq n$ be two non-negative integers or $m=\infty$, the restriction homomorphism $$\Ext_{\A(m)_*}^{s,t}(\H_*A_1)\rightarrow \Ext_{\A(n)_*}^{s,t}(\H_*A_1) $$ is an isomorphism if $t-s< f(s-1)+2^{n+1}-1$ and is an epimorphism if $t-s< f(s)+2^{n+1}$, where $\A(\infty)_*:=\A_*$ and $f(s)$ is as in Theorem \ref{Vanishing line}.
\end{Theorem}
\begin{proof} Pose $\Gamma = \A(m)_*\square_{\A(n)_*}\F$ and $\overline{\Gamma} = coker (\F\rightarrow \Gamma )$. The restriction homomorphism is the composite $$\Ext_{\A(m)_*}^{s,t}(\H_*A_1)\rightarrow \Ext_{\A(m)_*}^{s,t}( \Gamma\otimes \H_*A_1)\cong \Ext_{\A(n)_*}^{s,t}(\H_*A_1)$$ where the first map is induced by the unit $\F\rightarrow \Gamma$ and the second is the change-of-rings isomorphism.
The short exact sequence of $\A(m)_*$-comodules $\F\rightarrow \Gamma\rightarrow \overline{\Gamma}$ gives rise to a long exact sequence
$$\Ext_{\A(m)_*}^{s-1,t}( \overline{\Gamma}\otimes \H_*A_1)\rightarrow \Ext_{\A(m)_*}^{s,t}(\H_*A_1)\rightarrow \Ext_{\A(n)_*}^{s,t}(\H_*A_1)$$$$\rightarrow \Ext_{\A(m)_*}^{s,t}(\overline{\Gamma}\otimes \H_*A_1)$$
Because $\overline{\Gamma}$ is $2^{n+1}$-connected and $\Ext_{\A(m)_*}^{s,t}(\F,\H_*A_1)$ has the vanishing line $t-s<f(s)$, $$\Ext_{\A(m)_*}^{s,t}( \overline{\Gamma}\otimes \H_*A_1)=0$$ if $t-s< f(s) + 2^{n+1}$, hence the surjectivity of the respective restriction homomorphism;  and  $$\Ext_{\A(m)_*}^{s-1,t}(\overline{\Gamma}\otimes \H_*A_1)=\Ext_{\A(m)_*}^{s,t}( \overline{\Gamma}\otimes \H_*A_1)=0$$ if $t-s< f(s-1) + 2^{n+1}-1$, hence the bijectivity of the respective restriction homomorphism.
\end{proof}
\begin{Corollary} The restriction map $\Ext_{\A_*}^{s,t}(\H_*A_1)\rightarrow \Ext_{\A(2)_*}^{s,t}(\H_*A_1)$ is an epimorphism if $t-s<5s+8$ and is an isomorphism if $t-s<5s+2$.
\end{Corollary}
\begin{Theorem}\label{surjectivity of E2} The restriction map $\Ext_{\A_*}^{*,*}( \H_*(A_1))\rightarrow \Ext_{\A(2)_*}^{*,*}(\H_*(A_1))$ is an epimorphism.
\end{Theorem}
\begin{proof}  The restriction map $\mathrm{Res}: \Ext_{\A_*}^{*,*}(\H_*(A_1))\rightarrow \Ext_{\A(2)_*}^{*,*}(\H^*(A_1))$ is a map of modules over $\Ext_{\A_*}^{*,*}(\H_*(A_1\wedge DA_1))$; This module structure comes from the fact that $A_1$ is a module over the ring spectrum $A_1\wedge DA_1$. It is proved in \cite{BEM17}, Corollary 3.8 that the class $w_2\in \Ext_{\A(2)_*}^{s,t}(\H_*(A_1\wedge DA_1))$ lifts to $\Ext_{\A_*}^{s,t}(\H_*(A_1\wedge DA_1))$. In particular, the restriction $\mathrm{Res}$ is a map of modules over the sub-algebra $R$ generated by $g,\nu,w_2$. By Proposition \ref{AdamsA_1}, the classes $e_i$ where $$ i\in\{0, 5, 6, 11, 15, 17, 21, 23, 30, 32, 36, 38, 42, 47, 48, 53\} $$ generates $\Ext_{\A(2)_*}^{*,*}( \H^*(A_1)$ as a module over $R$. These classes live in the region $\{t-s<5s+8\}$, hence lift to $\Ext_{\A_*}^{*,*}( \H^*(A_1)$ by Theorem \ref{Approximation}.
\end{proof}
\noindent
This theorem shows that all the classes of $M\cup N$ (respectively $P\cup Q$) lift to $\Ext_{\A_*}^{*,*}(\H_*(A_1))$. 
In the following part, we prove that the latter lift to permanent cycles. 
\subsection{The topological $tmf$-Hurewicz homomorphism}
The key step is the following observation on $\Ext_{\A_*}^{*,*}(\H_*(A_1))$.

\begin{Theorem}\phantomsection\label{structure E2 A_1} The group $\Ext_{\A_*}^{*,*}(\H^*(A_1))$ has the following properties
\begin{itemize}
\item[(i)] All classes of $\Ext_{\A_*}^{s,t}( \H^*A_1)$ in the region $$F = \{s\geq 18, 5s\leq t-s\leq 5s+6\}\cup \{s\geq 27, 5s\leq t-s\leq 5s+14\} $$ are $g$-free and are divisible by $g$.
\item[(ii)] Any class $x $ of $\Ext_{\A_*}^{s,t}( \H^*A_1)$ in the region $$D =\{s\geq 21, 5s\leq t-s \leq 5s+12\}\cup \{s\geq 30, 5s\leq t-s\leq 5s+20\}$$ is weakly $g$-divisible, i.e., there is a class $y$ and a non-negative integer $n$ such that $g^{n+1}y = g^{n}x$. 
\end{itemize}
\end{Theorem}
\noindent
Because the classes involved in the statement of this theorem live in the region where there is an isomorphism $\Ext_{\A_*}^{s,t}( \H^*(A_1))\cong\Ext_{\A(4)_*}^{s,t}( \H^*(A_1))$ by the Approximation Theorem, it suffices to prove that $\Ext_{\A(4)_*}^{*,*}(\H^*A_1)$ has the required properties.
We prove a stronger statement:
\begin{Theorem}\label{structure E2} The group $\Ext_{\A(4)_*}^{*,*}(\H^*(A_1))$ has the following properties
\begin{itemize}
 \item[(i)] All classes in the region $$S_1=\{20\leq s\leq 27, 5s\leq t-s\leq 7s-40\} \cup \{s\geq 27, 5s\leq t-s\leq 5s+14\}$$ are $g$-free and are divisible by $g$. All classes in the region $$S_2=\{27\leq s\leq 30, 5s+14\leq t-s\leq 7s-40\}\cup$$$$\{s\geq 30, 5s+14\leq t-s\leq 5s+20\},$$ are weakly divisible by $g$.
 
 \item[(ii)] All classes in the region $$T_1=\{15\leq s\leq 18, 5s\leq t-s\leq 7s-30\} \cup \{s\geq 18, 5s\leq t-s\leq 5s+6\}$$ are $g$-free and are divisible by $g$. All classes in the region $$T_2=\{18\leq s\leq 21, 5s+6\leq t-s\leq 7s-30\}\cup$$$$\{s\geq 21, 5s+6\leq t-s\leq 5s+12\},$$ are weakly divisible by $g$.
\end{itemize}
\end{Theorem}
\begin{figure}[h!]
\begin{center}
	\begin{tikzpicture}
		\node (a) at (-1,0) {$s=20$};
		\node (b) at (-1, 3){$s=30$};
		\node (c) at (-1,2){$s=27$};
		\node[rotate=27] at (2.5,1.5) {$t-s =5s$};
		\node[rotate =18] at (6,1.5) {$t-s =7s-40$};
		\draw[-] (0,0)--(6,3);
		\draw[-] (0,0)--(10,3);
		\draw[-] (4,2)--(8/3+4,2);
		\draw[-] (8/3+4,2)--(6+8/3,3);
		\draw[-] (6,3)--(10,3); 
		\draw[fill=red] (0,0)--(8/3+4,2)--(4,2);
		\draw[fill=red] (8/3+4,2)--(4,2)--(6,3)--(6+8/3,3);
		\draw[fill=red] (6,3)--(6+8/3,3)--(8+8/3,4)--(8,4);
		\draw[fill=blue] (4+8/3,2)--(10,3)--(6+8/3,3);
		\draw[fill=blue](10,3)--(6+8/3,3)--(8+8/3,4)--(12,4);
	\end{tikzpicture}
	\caption{The region in red is associated to $S_1$ or $R_1$, the blue to $S_2$ or $R_2$.}
\end{center}
\end{figure}
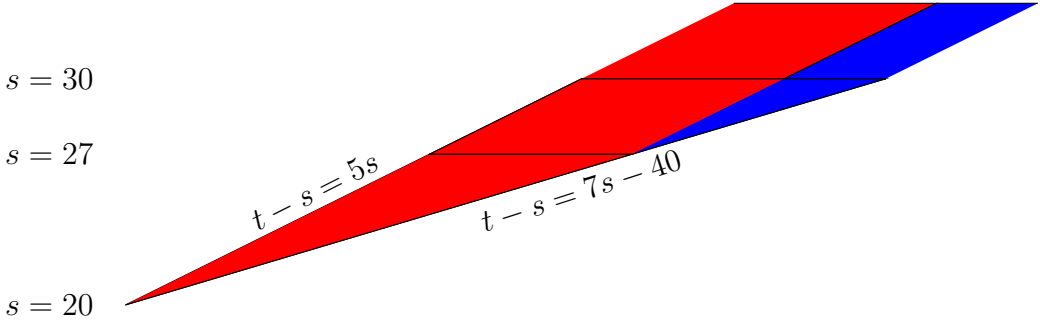
\noindent
Before proving this theorem, let us explain the strategy of the proof. Observe that there is a sequence of extensions of commutative Hopf algebras $$B_{i+1}\square_{B_{i}}\F\rightarrow B_{i+1}\rightarrow B_{i}\ \mbox{for}\ 0\leq i\leq 8$$ in which each $B_{i+1}\square_{B_{i}}\F$ is isomorphic to an exterior algebra $\Lambda(h_i)$ on one generator $h_i$ of degree at least $8$ and $B_0 = \A(2)_*$, $B_{9} = \A(4)_*$.  We can then deduce information on $\Ext_{\A(4)_*}^{*,*}(\H_*A_1)$ from $\Ext_{\A(2)_*}^{*,*}(\H_*A_1)$ by a sequence of Davis-Mahowald spectral sequences
$$\E_1^{s,t,\sigma}=	\underset{\sigma\geq 0}{\oplus}\Ext_{B_i}^{s,t-|h_i|\sigma}(\H_*A_1\otimes \F\{h_i^{\sigma}\})\Longrightarrow  \Ext_{B_{i+1}}^{s+\sigma,t}(\H_*A_1).$$
By the calculation of $\Ext_{\A(2)_*}^{*,*}(\H_*A_1)$, we see that the classes in the region $S_1$ and $S_2$ of the latter have the desired properties. Using this as the base case, we prove by induction that each $\Ext_{B_{i+1}}^{s+\sigma,t}(\H_*A_1)$ has the desired properties. To this end, we first prove, by induction on $r$, that the $\E_r$-term of the Davis-Mahowald spectral sequence has similar properties in the appropriate regions and then make sure that extensions cannot prevent the target of the spectral sequence from having the desired properties, where the fact that the degree of each $h_i$ is at least $8$ becomes crucial.
 
\begin{proof} We have that $$\A(4)_* = \F[\zeta_1,\zeta_2,\zeta_3,\zeta_4,\zeta_5]/(\zeta_1^{32},\zeta_2^{16},\zeta_3^{8},\zeta_4^{4},\zeta_5^{2})$$
$$\A(2)_* = \F[\zeta_1,\zeta_2,\zeta_3]/(\zeta_1^{8},\zeta_2^{4},\zeta_3^{2}).$$
From this, we can construct a sequence of maps of commutative Hopf algebras $(B_{i+1}\rightarrow B_i)$ with $0\leq i\leq 9$, $B_0 = A(2)_{*}$ and $B_9 = A(4)_{*}$ such that for each $i$, $B_{i+1}\square_{B_{i}}\F = \Lambda(h_i)$ is an exterior algebra on one generator $h_i$ of degree at least $8$. Informally, we start with $B_0 = A(2)_*$ and successively "add" $\zeta_4, \zeta_3^2, \zeta_2^4, \zeta_1^8, \zeta_5, \zeta_4^2, \zeta_3^4, \zeta_2^8, \zeta_1^{16}$. For example,
$$B_1 = \F[\zeta_1,\zeta_2,\zeta_3,\zeta_4]/(\zeta_1^{8},\zeta_2^{4},\zeta_3^{2},\zeta_4^2).$$
We will prove by induction on $i$ that $\Ext_{B_i}^{*,*}( \H_*A_1)$ has the property $(i)$. The proof of $(ii)$ works similarly, see Remark \ref{proof of ii}. First, we can directly check that $$\Ext_{B_0}^{*,*}(\H_*A_1)= \Ext_{\A(2)_*}^{*,*}(\H_*A_1)$$ verifies $(i)$, by inspecting its structure shown in Proposition \ref{AdamsA_1}. Suppose that $\Ext_{B_i}^{*,*}(\H_*A_1)$ verifies the properties $(i)$.
Consider the Davis-Mahowald spectral sequence
\begin{equation} \phantomsection\label{Davis-Mahowald SS}
	\E_1^{s,t,\sigma}=	\underset{\sigma\geq 0}{\oplus}\Ext_{B_i}^{s,t}(\H_*A_1\otimes \F\{h_i^{\sigma}\})\Longrightarrow  \Ext_{B_{i+1}}^{s+\sigma,t}( \H_*A_1)
\end{equation}
and the differential $d_r$ goes from $\E_r^{s,t,\sigma}$ to $\E_r^{s-r+1,t,\sigma+r}$.  Since $h_i$ is a $B_i$-primitive, we have that $$\E_1^{s,t,\sigma}=	\underset{\sigma\geq 0}{\oplus}\Ext_{B_i}^{s,t-d\sigma}(\H_*A_1)\otimes \F\{h_i^{\sigma}\}$$
where $d = |h_i|$.
We will prove by induction on $r\geq 1$ that each $\mathrm{E}_r^{s,t,\sigma}$-term of the Davis-Mahowald SS (\ref{Davis-Mahowald SS}) has the following properties \\

\begin{itemize}
\item[(a)] All classes in the region 
$$R_1 = \{20\leq s+\sigma\leq 27, 5(s+\sigma)\leq t-s-\sigma\leq 7(s+\sigma)-40\}$$ $$\cup\{s+\sigma\geq 27, 5(s+\sigma)\leq t-s-\sigma\leq 5(s+\sigma)+14\}$$ 
are $g$-free and are divisible by $g$. 
\item[(b)] All classes in the region 
$$R_2 = \{27\leq s+\sigma\leq 30, 5(s+\sigma)+14\leq t-s-\sigma\leq 7(s+\sigma)-40\}$$$$\cup\{s+\sigma\geq 30, 5(s+\sigma)+14\leq t-s-\sigma\leq 5(s+\sigma)+20\}$$ are weakly divisible by $g$.\\
\end{itemize}
\noindent
A similar proof as of Theorem \ref{Vanishing line} shows that $\Ext_{B_i}^{s,t}(\H_*A_1)$ has the same vanishing line, and so
\begin{align}\label{vanishing line E1}
\E_1^{s,t,\sigma}=0\
\mbox{if}\ s+\sigma >6, t-(s+\sigma)<5(s+\sigma)\ \mbox{or}\\ \nonumber
s+\sigma\leq 6, t-(s+\sigma)<5(s+\sigma)-4.
\end{align}
The $\E_1^{s,t,\sigma}$-term is spanned by classes $x\otimes h_i^{\sigma}$ with $x\in \Ext_{B_{i}}^{s, t-d\sigma} (\H_*A_1)$. For degree reasons ($d=|h_i|\geq 8$) and Equation (\ref{vanishing line E1}), a class $x\otimes h_i^{\sigma}$ living in $R_1$ and $R_2$ is non-trivial only if $x$ lies in $S_1$ and $S_1\cup S_2$, respectively. Then together with the induction hypothesis, it is straightforward to see that the $\mathrm{E}_1$-term of the spectral sequence (\ref{Davis-Mahowald SS}) has the properties $(a)$ and $(b)$. Suppose that the $\mathrm{E}_r$-term of (\ref{Davis-Mahowald SS}) has those properties. Let $x\in \E_r^{s,t,\sigma}$ represent a class $[x]$ of $\E_{r+1}^{s,t,\sigma}$.

\textbf{Step 1.} Suppose $[x]$ lives in $R_1$ and $[x]$ is $g$-torsion. Because $R_1$ is stable by multiplication by $g$, we can assume that $g[x]=0$. Then there exists $y\in \E_r^{s+4+r-1, t+24, \sigma-r}$ such that $d_r(y) = gx$. By an inspection on degrees, we see that $y$ belongs to the region $R_1\cup R_2$. By the induction hypothesis, $y$ is weakly divisible by $g$, i.e., there is a $z$ and an integer $n$ such that $g^{n+1} z = g^n y$. It follows that 
$$g^{n+1}d_r(z) = d_r(g^{n+1}z) = d_r(g^ny) = g^{n} d_r(y) = g^ngx= g^{n+1}x.$$ 
However, $d_r(z)-x$ lies in $R_1$ which consists only of $g$-free classes, hence $d_r(z)=x$, and so $[x]=0$. Therefore, all classes in $R_1$ of $\E_{r+1}$ are $g$-free.

\textbf{Step 2.} Suppose $[x]$ lies in $R_1$. By the induction hypothesis, there exists $y\in \E_r^{s-4,t-24,\sigma}$ such that $gy=x$. We claim that $y$ is a $d_r$-cycle. We have that $$gd_r(y)= d_r(gy)=d_r(x) = 0.$$ Moreover $d_r(y)\in E_r^{s-r-3, t-24,\sigma+r}$ living in $R_1$, hence is $g$-free. We conclude that $d_r(y) = 0$. Thus, $[x]$ is divisible by $g$.

Step 1 and Step 2 show that the $\E_{r+1}$-term has the property $(a)$.

\textbf{Step 3.} Now suppose that $[x]$ belongs to the region $R_2$. Then $x$ is weakly divisible by $g$, i.e., there is a class $z\in \E_r^{s-4,t-24,\sigma}$ and an integer $n$ such that $g^{n+1} z = g^{n}x.$ We claim that $z$ is a $d_r$-cycle. Since $x$ is a $d_r$-cycle, we have $$g^{n+1}d_r(z) = d_r(g^{n+1}z) = d_r(g^{n}x) = g^{n}d_r(x) = 0.$$
% (s(x) means s+\sigma of $x$)Suppose $27\leq s(x)\leq 30 $ and $t(x)-s(x) \leq 7s(x) -40$. Then $s(z) = s(x)-4$ and $t(z)-s(z) = t(x)-s(x)-20$. And so $24\leq s(y) = s(x)-3\leq 27$ and $t(y)-%s(y) = t(z)-s(z)-1 = t(x)-s(x)-21 \leq 7s(x) -40 -21 = 7(s(x)-3) -40 = 7s(y) -40$
% Suppose  s(x)\geq 30, t(x)-s(x)\leq 5s(x) +20. So s(z) =s(x) -4, t(z)-s(z) = t(x)-s(x)-20. So s(y) =s(x)-3/geq 27, t(y)-s(y) =t(z)-s(z) -1= t(x)-s(x)-21\leq 5s(x)-21+20 = 5s(x) -1 = 5s(y)+15-1 =5s(y)+14
Moreover, $d_r(z)\in \E_r^{s-4-r+1,t-24, \sigma+r}$ which belongs to $R_1$, hence $d_r(z)$ is $g$-free, and so $d_r(z)=0$. Therefore, we obtain that $g^{n+1}[z] = g^{n}[x],$ hence the $\mathrm{E}_{r+1}$-term has the property $(b)$.

\textbf{Step 4.} It is now straightforward to see that the $\mathrm{E}_{\infty}$-term also has the properties $(a)$ and $(b)$. To finish the proof, we will show that the target of the spectral sequence (\ref{Davis-Mahowald SS}) has the property $(i)$. Let $$...\subset F_{\sigma}\subset F_{\sigma-1}\subset ...\subset F_1\subset F_0=\Ext_{B_{i+1}}^{*,*}(\F,\H_*A_1)$$ be the filtration of $\Ext_{B_{i+1}}^{*,*}(\H_*A_1)$ associated to the Davis-Mahowald spectral sequence. A class belongs to $F_{\sigma}$ only if it is represented in the $\E_1$-term by a class of the form $x\otimes h_i^\sigma$ where $x \in \Ext_{B_{i}}^{s, t} ( \H_*A_1)$ - here $t-s\geq 5s - 4$ because of Equation \ref{vanishing line E1}. Such a class has bidegree $(s+\sigma, t+d\sigma)$, and so has the topological degree $t-s+(d-1)\sigma$. Because $d\geq 8$, the latter exceeds $5(s+\sigma) +20$ for $\sigma$ sufficiently large. However, any class in $S_1\cup S_2$ has bidegree $(t,s)$ satisfying $ t-s\leq  5s + 20$. This means that there is an integer $m$ such that all classes of $S_1\cup S_2$ belongs to $F_0\backslash F_{m}\cup \{0\}$. From this, it is straightforward to verify that the properties $(a)$ and $(b)$ of the $\E_\infty$-term implies the property $(i)$ of $\Ext_{B_{i+1}}^{*,*}(\F,\H_*A_1)$.
\end{proof}
\begin{Remark}\label{proof of ii} The proof of $(i)$ uses a double induction. What makes both base cases work is the fact that $\Ext_{\A(2)_*}^{s,t}(\H_*A_1)$ has the required properties and that the slope of each $h_i$ is lower than $\frac{1}{7}$ which is exactly the slope of the lower line limiting the region in question. What makes the inductive step work is self-explained by the choice of the regions: relevant classes lie in the relevant regions. What makes the target of the DMSS have the required properties is that the slope of $h_i$ is lower then the slope of the vanishing line. The regions $T_1$ and $T_2$ are chosen to have all of these features, hence the proof of $(ii)$ is similar to that of $(i)$.
\end{Remark}
\noindent
We need the following lemma on the $\E_2$-term of the Adams spectral sequence for $S^0$, which necessitates a calculation of the $\Ext$ group up to stem 43, see \cite{Tan70}, Theorem 4.42.
\begin{Lemma} The class $\nu$ is annihilated by $g^2$, so is $g$-torsion in the $\E_2$-term of the ASS for $S^0$. 
\end{Lemma}

\begin{Theorem}\label{surj_Hurw} The induced map in homotopy of the Hurewicz map $A_1\rightarrow tmf\wedge A_1$ is surjective.
\end{Theorem}
\begin{proof} The map $H_*: \pi_*(A_1)\rightarrow \pi_*(tmf\wedge A_1)$ is a map of $\pi_*(A_1\wedge DA_1)$-modules. In particular, it is a map of modules over the subalgebra $R$ of $\pi_*(A_1\wedge DA_1)$ generated by $\nu, \overline{\kappa}, v_2^{32}$. Therefore, we only need to prove that a set of generator of $\pi_{*}(tmf\wedge A_1)$ as a $R$-module belongs to the image of $H$.
\\\\
Because of Theorem \ref{surjectivity of E2}, we can choose a lift of classes of $M\cup N$ (respectively $P\cup Q$) to $\Ext_{\A_*}^{*,*}(\H_*A_1)$ such that classes which are divisible by $\nu$  lift to classes which are divisible by $\nu$. We fix such a choice of lifting and call them also $M$ and $N$ (respectively, $P$ and $Q$). Let us prove that all classes of $M\cup N$ (respectively, $P\cup Q$) are permanent cycles in the ASS for $A_1$; then they must survive to the $\mathrm{E}_{\infty}$-term because their images in the ASS for $tmf\wedge A_1$ do.  By comparing the bidegree of the classes of $M\cup N$ (respectively $P\cup Q$) and the vanishing line of the $\E_2$-term, we see that the differentials on classes of $M\cup N$ (respectively $P\cup Q$) of length greater than $4$ are trivial. The theorem now follows from Proposition \ref{petit stem}, \ref{non-self dual}, \ref{self dual} below. 
\end{proof}
\begin{Proposition}\label{petit stem} All classes in $M$ of $A_1[00]$ and $A_1[11]$ and in $P$ of $A_1[01]$ and $A_1[10]$ are permanent cycles in the respective ASS.
\end{Proposition}
\begin{proof} Inspection of bidegrees together with the Vanishing Line theorem \ref{Vanishing line} show that there can only be nontrivial differentials $d_2$ on classes of $M$ (respectively $P$) and moreover these differentials hit the region where there is an isomorphism between $\Ext_{\A_*}^{*,*}( \H_*(A_1))$ and $\Ext_{\A(2)_*}^{*,*}( \H_*(A_1))$. However, all classes of M (respectively $P$) are permanent cycles in the ASS for $tmf\wedge A_1$. Therefore, the differentials $d_2$ on the classes in $M$ (respectively $P$) in the ASS for $A_1$ are trivial.
\end{proof}

\begin{Lemma} \label{free E3} a) The target groups for $d_3$ on classes in $N$ are $g$-free. More precisely, $\E_3^{s,t}$ is $g$-free if 
$$(s,t)\in F_3:= \{s\geq 23, 5s\leq t-s\leq 5s+1\}\cup\{s\geq 28, 5s\leq t-s\leq 5s+9\}.$$
b) Suppose $s\geq 30$ and $t-s\leq 5s+20$ and let $s\in \E_3^{s,t}$. Then $x$ is weakly divisible by $g$, i.e., there exists an integer $n$ and a class $y\in E_3^{s-4, t-24}$ such that $g^{n+1}y = g^{n}x$.  
\end{Lemma}
\begin{proof}
a) The differential $d_2$-arriving in $g\E_2^{s,t}$ with $(s,t)\in F_3$ starts in $\E_2^{s',t'}$ with
$$(s', t') = (s+2, t+23), i.e., (t'-s')= t-s+21.$$
Then we have $$s'\geq 25, 5s' = 5s+10\leq t-s +10\leq 5s +11= 5s'+1$$
respectively, $$s'\geq30, 5s'=5s+10\leq t-s+10\leq 5s+19= 5s'+9.$$
So $(s', t')$ belongs to 
$$\{s'\geq 25, 5s'+11\leq t'-s'\leq 5s'+12\}\cup\{s'\geq 30, 5s'+11\leq t'-s'\leq 5s'+20\}.$$
In this region, Theorem \ref{structure E2 A_1} guarantees that all classes are weakly divisible by $g$ and this implies that $\E_3^{s,t}$ is $g$-free if $(s,t)\in F_3$. In effect, suppose $x$ is a class which lies in $F_3$ and which is $g$-torsion. The region $F_3$ is stable by multiplication by $g$, so we can assume that $gx=0$. Let $a$ be a representative of $x$. Then there exists $b\in \E_2$ such that $d_2(b) = ga$. By the argument above, we know that $a$ is weakly divisible by $g$, i.e., there exists an integer $n\geq 0$ and a class $c$ at the $\E_2$-term such that $g^na = g^{n+1}c$. Then we have
$$g^{n+1}d_2(c) = d_2(g^{n+1}c) = d_2(g^{n}b) =g^nd_2(b) = g^nga=g^{n+1}a.$$
However, $F_3$ belongs to the region which is $g$-free at the $\E_2$-term, hence $d_2(c) = a$, which means that $x=0$ at the $\E_3$-term.

b) We represent $x$ by a $d_2$-cycle $a$. By part $(ii)$ of Theorem \ref{structure E2 A_1}, there exists $b\in \E_2^{s-4, t-24}$ and an integer $n$ such that $g^na = g^{n+1}b$. It is enough to show that $b$ is a $d_2$-cycle. We first note that, since $a$ is a $d_2$-cycle, we have $$g^{n+1}d_2(b) = d_2(g^{n+1}b) = d_2(g^na) = 0.$$ Therefore it is enough to show that $d_2(b)$ is $g$-free. In fact, $d_2(b)$ is a class in $\E_2^{s',t'}$ with $s'=s-2$ and $$t'-s' =t-s - 19\leq 5s+20-19= 5(s-2)+11 = 5s'+11,$$ so it is $g$-free by Theorem \ref{structure E2 A_1} part $(ii)$
\end{proof}
\begin{Lemma}\label{free E4} The target groups for the differential $d_4$ on classes in N are $g$-free. More precisely, $\E_4^{s,t}$ is $g$-free if 
$$(s,t)\in F_4:= \{s\geq 29, 5s\leq t-s\leq 5s+4\}$$
\end{Lemma}
\begin{proof}The differential $d_3$ arriving in $g\E_3^{s,t}$ with $(s,t)\in F_4$ starts in $E_3^{s',t'}$ with 
$$s'=s+1, t'-s'=t-s+21.$$Then we have $$s'\geq 30, 5s'+16 \leq t'-s'\leq 5s'+20.$$ By Lemma \ref{free E3}, all classes in such bidegrees are weakly divisible by $g$ and this implies that $\E_4^{s,t}$ is $g$-free if $(s,t)\in F_4$. In effect, suppose $x$ is a class which lies in $F_4$ and which is $g$-torsion. Because $F_4$ is stable by multiplication by $g$, we can assume that $gx=0$. Let $a\in \E_3^{s,t}$ be a representative of $x$. Then there exists $b\in \E_3$ such that $d_3(b)=ga$. By the argument above, $b$ is weakly divisible by $g$, i.e., there is an non-negative integer $n$ and a class $c\in\E_3$ such that $g^{n+1}c=g^{n}b$. Then we have $$g^{n+1}d_3(c) = d_3(g^{n+1}c) = d_3(g^{n}b)=g^nd_3(b)=g^nga=g^{n+1}a.$$
However, $F_4$ belongs to the region where $g$ acts freely at the $\E_3$-term by Lemma \ref{free E3} part $(i)$, hence $d_3(c)=a$ and so $x=0$ at the $\E_4$-term.
\end{proof}

\begin{Proposition} \label{non-self dual}The differentials $d_2$, $d_3$, $d_4$ on the classes in N for $A_1[00]$ and $A_1[11]$ are trivial. 
\end{Proposition}
\begin{proof}  All classes of $N$ are divisible by $\nu$, so are $g$-torsion in the $\E_2$-term, hence are $g$-torsion at all terms. It is then enough to show that the target groups of differentials $d_2$, $d_3$, $d_4$ on the classes in N are $g$-free at the $\E_2$, $\E_3$, $\E_4$-terms, respectively. In effect, the target groups for the differential $d_2$ on the classes in N lie in the region $$\{s\geq 19, 5s\leq t-s\leq 5s+6\}\cup\{s\geq27, 5s\leq t-s\leq5s+14\}$$ consisting only of $g$-free classes, by Theorem \ref{structure E2 A_1}$(i)$. Next, a potential nontrivial differential $d_3$ or $d_4$ on the classes in N has its target in the region $F_3$ or $F_4$, respectively, which is $g$-free by Lemma \ref{free E3} or Lemma \ref{free E4}, respectively. 
\end{proof}
\begin{Proposition} \label{self dual}The differentials  $d_2$, $d_3$, $d_4$ on the classes in Q for $A_1[10]$ and $A_1[01]$ are trivial.
\end{Proposition}
\begin{proof} In this proof, $A_1$ denotes the self dual versions $A_1[10]$ and $A_1[01]$.
The same argument as in the proof of Proposition \ref{non-self dual} shows that the differentials  $d_2$, $d_3$, $d_4$ on the classes in N which are divisible by $\nu$ are trivial. Consider the four other classes in $N$
\begin{equation}\label{g free in N}
w_2^2e_0, w_2^2e_5, w_2^2e_{15}, w_2^2e_{30}. 
\end{equation}
These classes are $g$-free at the $\E_2$-term. We now show that their $g$-multiple towers are truncated by differentials $d_2$ in the ASS for $A_1$. In effect, by Theorem 4.0.3 of \cite{Pha18}, the following differentials $d_2$ happen in the ASS for $tmf\wedge A_1$ 

$$d_2(w_2 e_{53}) = g^5 e_0, \ d_2(w_2e_{38}) = g^4e_5,$$
$$d_2(w_2e_{48})= g^4e_{15}, \ d_2(w_2e_{23}) = g^2e_{30}.$$
Since the targets of these differentials live on the line $t-s = 5s$ where the $\E_2$-term of the ASS for $A_1$ and that for $tmf\wedge A_1$ are isomorphic, the same differentials happen in the ASS for $A_1$. Moreover, the class $w_2^2$ is a cycle for the differential $d_2$ in the ASS for $A_1\wedge DA_1$ as shown the proof of Lemma 3.10 of \cite{BEM17}. Therefore, by the Leibniz rule, the following differentials $d_2$ happen in the ASS for $A_1$:
$$d_2(w_2^3 e_{53}) = g^5w_2^2 e_0, \ d_2(w_2^3e_{38}) = g^4w_2^2e_5,$$
$$d_2(w_2^3e_{48})= g^4w_2^2e_{15}, \ d_2(w_2^3e_{23}) = g^2w_2^2e_{30}.$$
It follows that the differentials $d_2$ on the classes of (\ref{g free in N}) are $g$-torsion. Moreover, the differentials $d_2$ on the latter arrive in $\E_2^{s,t}$ with $s\geq 18, 5s\leq t-s\leq 5s+6$ which consists only of $g$-free classes by Theorem \ref{structure E2 A_1}. Thus, the classes of \ref{g free in N} are $d_2$-cycles and become $g$-torsions in the $\E_3$-term.   
\\\\
The differentials $d_3$ on the classes of ($\ref{g free in N}$) arrive in $\E_3^{s,t}$ with $s\geq 19, t-s = 5s$. For these bidegrees, there is an isomorphism at the $\E_2$-term of the ASS for $A_1$ and that for $tmf\wedge A_1$; in particular, the related $\Ext$-groups are isomorphic to $\F$ and are generated by $g^4e_{15}$, $g^5e_0$, $g^4e_{30}$, $g^5e_{5}$, respectively. However, in the ASS for $tmf\wedge A_1$, by Theorem 4.0.3 of \cite{Pha18}, the latter are hit by differentials $d_2$, which are
$$d_2(w_2 e_{53}) = g^5 e_0, \ d_2(gw_2e_{38}) = g^5e_5,$$
$$d_2(w_2e_{48})= g^4e_{15}, \ d_2(g^2w_2e_{23}) = g^4e_{30}.$$  Because of Theorem \ref{surjectivity of E2} and the naturality of the ASS, $\E_3^{s,t} = 0$ for $s\geq 19, t-s =5s$ in the ASS for $A_1$. Thus, the differentials $d_3$ on the classes of (\ref{g free in N}) are trivial. 
\\\\
Finally, the differential $d_4$ on the classes of (\ref{g free in N}) land above the vanishing line, hence are trivial.
\end{proof}
\begin{Remark} To illustrate the proof of Proposition \ref{self dual}, we give some examples and more details. 
\\\\
a) First, a differential $d_3$ or $d_4$ on the first five classes in N listed in Table 2 has target living above the vanishing line, so it is trivial. 
\\\\
b) The other classes in N might support non-trivial differentials $d_3$. For example, a differential $d_3$ on the class $\nu w_2^2 e_{17}$ arrives in $\E_3^{s,t}$ with $s= 23, t-s = 115= 5s$; the latter group is $g$-free by Lemma \ref{free E3}. The worst case is the class $\nu w_2^3e_{23}$ on which a differential $d_3$ lives in $\E_3^{s,t}$ with $s=32, t-s = 169=5s+9$, which is $g$-free by Lemma \ref{free E3}. 
\\\\
c) Only the last eight classes in N as listed in Table 2 might support non-trivial differentials $d_4$. These classes lie in $\E_4^{s,t}$ with $s\geq 25, 5s\leq t-s\leq 5s+25.$ Then $d_4$ on these arrives in $\E_4^{s',t'}$ with $s'=s+4$ and $t'=t+3$, and so $$s'\geq 29, 5s'\leq t'-s'\leq 5s'+4.$$ This region consists only of $g$-free classes by Lemma \ref{free E4}
\end{Remark}
\subsection{The edge homomorphism of the topological duality spectral sequence}\label{TDSS}
\noindent
In this last section, we prove Theorem \ref{edge-homo}. Let us restate the theorem.
\begin{Theorem}\label{edge-proof} The edge homomorphism of the topological duality spectral sequence $$\pi_*(E_C^{h\mathbb{S}^1_C}\wedge A_1)\rightarrow \pi_*(E_C^{hG_{24}}\wedge A_1)$$ is surjective. Therefore, all differentials starting from the $0$-line of the topological duality spectral sequence are trivial.
\end{Theorem} 
\noindent
Here, $E_C$ is the Lubin-Tate spectrum associated to the formal completion $F_C$ of the supersingular elliptic curve $C: y^2 +y = x^3$ defined over $\FF_4$. Let $\S_C$ be the automorphism group of $\S_C$. The group $\S_C^1$ is defined to be the kernel of the reduced determinant map $\S_C\rightarrow \Z_2$. The automorphism group of $C$ is isomorphic to $G_{24}: Q_{8}\rtimes C_3$, where $Q_8$ is the quaternion group and $C_3$ is the cyclic group of order $3$, and $G_{24}$ naturally embeds into $\S_C^1$.
We refer the reader to \cite{BG18} for the construction of the topological duality resolution and to \cite{Pha18} for more motivations on the study of this spectral sequence for $A_1$. 
\\\\
The duality spectral sequence has four lines and converges to $\pi_*(E_C^{h\S^1_C}\wedge A_1)$. Its edge homomorphism $E_C^{h\S^1_C}\rightarrow E_C^{hG_{24}}$ is induced by the inclusion of subgroup $G_{24}\rightarrow \S^1_C$. It suffices to prove that the map $$\pi_*(E_C^{h\S_C}\wedge A_1)\rightarrow \pi_*(E_C^{hG_{24}}\wedge A_1)$$ induced by the inclusion of subgroup $G_{24}\rightarrow \S_C$ is surjective. Let $\Gal$ denote the Galois group of $\FF_4$ over $\F$. The latter acts on $\S_C$ and $G_{24}$, because $C$ is already defined over $\F$. By applying Lemma 1.37 of \cite{BG18}, there are homotopy equivalences $$\Gal_+\wedge E_C^{h(\S_C\rtimes \Gal)}\simeq E_C^{h\S_C}$$ and $$\Gal_+\wedge E_C^{h(G_{24}\rtimes \Gal)}\simeq E_C^{hG_{24}},$$ that fit into a commutative diagram
$$\xymatrix{ \Gal_+\wedge E_C^{h(\S_C\rtimes \Gal)}\ar[d] \ar[r]& \Gal_+\wedge E_C^{h(G_{24}\rtimes \Gal)}\ar[d]\\
E_C^{h\S_C} \ar[r]& E_C^{hG_{24}},
}$$
where horizontal maps are induced by respective inclusions of subgroup. By Devinatz-Hopkins in \cite{DH04}, $$L_{K(2)}S^0\simeq E_C^{h(\S_C\rtimes \Gal)}$$ and so the map $E_C^{h(\S_C\rtimes \Gal)}\rightarrow E_C^{h(G_{24}\rtimes \Gal)}$ is identified with the unit map $L_{K(2)}S^0\rightarrow E_C^{h(G_{24}\rtimes \Gal)}$. By \cite{DFHH14}, the latter factorises through the homotopy equivalence $$L_{K(2)}tmf\simeq E_C^{h(G_{24}\rtimes \Gal)}.$$
Therefore, it is reduced to show that the map $L_{K(2)}A_1\rightarrow L_{K(2)}(tmf\wedge A_1)$ induces a surjection on homotopy. In fact, the latter fits in the following commutative diagram
\begin{equation}\label{diagr}
\xymatrix{A_1 \ar[r]\ar[d] &tmf\wedge A_1\ar[d]\\
   [(v_2^{32})^{-1}] A_1 \ar[r] \ar[d]& [(v_2^{32})^{-1}] (tmf\wedge A_1) \ar[d] \\
   L_{K(2)} A_1 \ar[r] & L_{K(2)} (tmf\wedge A_1).
 }  
\end{equation}
We show in Theorem 5.1.2 of \cite{Pha18} that the natural map $[(\Delta^8)^{-1}](tmf\wedge A_1)\rightarrow L_{K(2)}(tmf\wedge A_1)$ is a homotopy equivalence. In addition, by the proof of Lemma \ref{multi v2}, $v_2^{32}$ is equal to $\Delta^8$ up to some power. Therefore, the map $[(v_2^{32})^{-1}] (tmf\wedge A_1) \rightarrow  L_{K(2)} (tmf\wedge A_1)$ in the diagram (\ref{diagr}) is a homotopy equivalence. On the other hand, the induced map in homotopy of the middle map of (\ref{diagr}) is identified with a direct limit of (\ref{hurew-map})
$$(v_2^{32})^{-1}\pi_*(A_1)\rightarrow (v_2^{32})^{-1}\pi_*(tmf\wedge A_1), $$ hence it is a surjection, because (\ref{hurew-map}) is so. Therefore, the composite $$[(v_2^{32})^{-1}]  A_1\rightarrow [(v_2^{32})^{-1}] (tmf\wedge A_1) \rightarrow  L_{K(2)} (tmf\wedge A_1)$$ induces a surjection in homotopy, hence so is the map $$L_{K(2)}A_1\rightarrow L_{K(2)}(tmf\wedge A_1),$$ because of the commutative diagram (\ref{diagr}). We are done.
\bibliographystyle{alpha}
\bibliography{mybibs}{}
\end{document}